\documentclass[pdflatex,sn-mathphys-num]{sn-jnl}

\usepackage[T1]{fontenc}
\usepackage{mathtools}
\usepackage{amssymb,amsfonts}
\usepackage{amsthm}
\usepackage{enumitem}
\usepackage{textcomp}

\setcitestyle{numbers,square,sort&compress}
\hypersetup{
  hidelinks,
  pdftitle={Convex order under majorization for projections of ell-p balls},
  pdfauthor={Soufiane Fafe},
  pdfkeywords={ell-p balls, convex order, majorization, stochastic orders, hyperplane sections, Gaussian convolution}
}

\numberwithin{equation}{section}
\allowdisplaybreaks

\theoremstyle{thmstyleone}
\newtheorem{theorem}{Theorem}[section]
\newtheorem{proposition}[theorem]{Proposition}
\newtheorem{lemma}[theorem]{Lemma}
\newtheorem{corollary}[theorem]{Corollary}

\theoremstyle{thmstylethree}

\newtheorem{remark}[theorem]{Remark}

\newcommand{\R}{\mathbb{R}}
\newcommand{\E}{\mathbb{E}}
\DeclareMathOperator{\Var}{Var}
\DeclareMathOperator{\vol}{vol}
\newcommand{\1}{\mathbf{1}}
\newcommand{\dd}{\,\mathrm{d}}
\newcommand{\eps}{\varepsilon}
\newcommand{\cx}{\preceq_{\mathrm{cx}}}
\newcommand{\sphere}{S^{n-1}}
\newcommand{\diag}{u_{\mathrm{diag}}}
\newcommand{\sqvec}[1]{#1^{\odot 2}}

\begin{document}

\title[Convex order under majorization]
 {Convex order under majorization for projections of $\ell_p^n$-balls}

\author*[1]{\fnm{Soufiane} \sur{Fafe}}\email{sf3159@columbia.edu}

\affil*[1]{\orgdiv{Department of Industrial Engineering and Operations Research},
 \orgname{Columbia University},
 \orgaddress{\street{500 W. 120th Street \#315}, \city{New York},
 \state{NY}, \postcode{10027}, \country{USA}}}

\abstract{Let $n\ge1$ and $1\le p\le2$, and let $X$ be uniformly distributed on the unit ball $B_p^n$ of $\ell_p^n$. We prove that if $(a_1^2,\ldots,a_n^2)$ majorizes $(b_1^2,\ldots,b_n^2)$, then $\langle b,X\rangle^2$ is dominated by $\langle a,X\rangle^2$ in convex order. For $p<2$ this gives a nontrivial comparison of the complete squared-projection laws, whereas at $p=2$ it reduces to equality by rotational invariance. The result strengthens known majorization inequalities for individual power moments and provides a single distributional source for several familiar projection inequalities. Its proof combines a two-dimensional hinge-function comparison with a preservation theorem for Beta--Rademacher perturbations, which lifts elementary majorization transfers through the Dirichlet representation of the uniform measure on $B_p^n$. The comparison is stable under a common independent radial scaling and therefore extends to $\ell_p$-radial measures with nonincreasing radial density. Consequences include simultaneous inequalities for moments, hinge functions, and Laplace transforms, together with Schur-convexity of Gaussian-regularized central-section densities. The range is sharp in the stated direction: for $p>2$, a support obstruction already occurs between a coordinate direction and a two-coordinate direction.}



\maketitle

\section{Introduction}\label{sec:introduction}

Let $n\ge1$. For $p>0$, set
\[
 B_p^n=\left\{x\in\R^n:\sum_{i=1}^n |x_i|^p\le1\right\},
\]
and let $X$ be uniformly distributed on $B_p^n$. For
$a=(a_1,\ldots,a_n)\in\R^n$, write
\[
 \sqvec{a}=(a_1^2,\ldots,a_n^2).
\]
If $u,v\in\R^n$ have the same coordinate sum, we write $u\succ v$ when
$u$ majorizes $v$, that is, when the partial sums of the decreasing
rearrangement of $u$ dominate those of $v$. Because the law of $X$ is
unconditional and invariant under coordinate permutations, majorization
of $\sqvec a$ is the natural order for measuring how concentrated the
Euclidean mass of the coefficient vector is among its coordinates.

Our main result orders the entire squared-projection laws in convex
order. Recall that, for integrable random variables $U,V$, the relation
$U\cx V$ means
\[
 \E\Phi(U)\le\E\Phi(V)
\]
for every convex function $\Phi$ for which the expectations exist.

\begin{theorem}[Convex order for squared projections]\label{thm:convex-order-main}
Let $n\ge1$ and $1\le p\le2$, let $X$ be uniformly distributed on
$B_p^n$, and let $a,b\in\R^n$. If
\[
 \sqvec a\succ\sqvec b,
\]
then
\begin{equation}\label{eq:convex-order-main}
 \langle b,X\rangle^2\cx\langle a,X\rangle^2.
\end{equation}
\end{theorem}

Majorization includes equality of the coordinate sums, so the hypothesis
implies $\|a\|_2=\|b\|_2$. Consequently, the two random variables in
\eqref{eq:convex-order-main} have the same mean. Equivalently, the
theorem compares all hinge functions:
\[
 \E\bigl(\langle b,X\rangle^2-t\bigr)_+
 \le
 \E\bigl(\langle a,X\rangle^2-t\bigr)_+,
 \qquad t\ge0.
\]
Thus it is stronger than either a comparison of one prescribed moment or
a statement about extremal central sections.

Majorization inequalities for individual power moments of projections
of $B_q^n$, $0<q\le2$, were obtained by Eskenazis, Nayar, and Tkocz
\cite{EskenazisNayarTkocz}. Related distributional stability refinements
of the Szarek and Ball inequalities were established by the same authors
in \cite{EskenazisNayarTkoczStability}. The former moment theorem covers
the larger range $0<q\le2$; in the non-Euclidean overlap
$1\le p<2$, Theorem~\ref{thm:convex-order-main} is stronger in the
type of comparison, replacing the family of power tests by every convex
function of the squared projection. At $p=2$, all equal-norm projections
have the same law. In particular, the theorem simultaneously yields the
moment inequalities in Corollary~\ref{cor:moments-laplace} and the
non-power Laplace-transform comparison
\[
 \E e^{-\lambda\langle b,X\rangle^2}
 \le
 \E e^{-\lambda\langle a,X\rangle^2},
 \qquad \lambda>0.
\]
The latter is the input for the Gaussian-regularized section
inequalities in Section~\ref{sec:applications}.

A classical line of work on stochastic majorization begins with a
result of Marshall and Proschan relating coefficient
majorization to convex inequalities for linear combinations of
exchangeable random variables; see also the subsequent theory of
stochastic majorization developed by Nevius, Proschan, and Sethuraman
\cite{MarshallProschan,NeviusProschanSethuraman}. Theorem~\ref{thm:convex-order-main} is not a direct instance of this principle.
Here the majorization hypothesis is imposed on the squared coefficient
vector $\sqvec a$, whereas the random variable being ordered is the
nonlinear quantity
\[
 \left(\sum_{i=1}^n a_iX_i\right)^2.
\]
Moreover, the coordinates of a uniform point in $B_p^n$ are dependent.
The Beta--Rademacher preservation theorem is the mechanism that allows
the two-coordinate comparison to be lifted through this dependence.

Majorization of squared coefficients also appears in sharp
Khinchin-type inequalities. Baernstein and Culverhouse established
majorization comparisons for sums of independent uniform spherical and
ball variables under bisubharmonic tests
\cite{BaernsteinCulverhouse}. Related inequalities for sums of
independent uniform random variables were obtained by Lata{\l}a and
Oleszkiewicz \cite{LatalaOleszkiewicz}. Negative-moment comparisons
and their connection with slicing were developed by Chasapis,
K{\"o}nig, and Tkocz \cite{ChasapisKonigTkocz}, while Melbourne and
Roberto obtained slicing inequalities from a different
transport-majorization comparison between densities in convex order
\cite{MelbourneRoberto}. In contrast, Theorem~\ref{thm:convex-order-main} compares the complete squared-projection
laws generated by a single dependent sign--Dirichlet vector.

The extremal central-section problem itself is already understood in
the range $0<p<2$: coordinate directions maximize and diagonal
directions minimize, with a Schur-convex formulation and quantitative
stability refinements
\cite{MeyerPajor,Koldobsky1998,Eskenazis,ChasapisNayarTkocz}.
Theorem~\ref{thm:convex-order-main} is instead a distributional
strengthening of the endpoint theory. A central-section
density records the behavior of one projection law at the single point
zero, whereas convex order controls every convex integral functional of
the squared projection. In Section~\ref{sec:applications}, the
non-strict classical section comparison is recovered from the theorem
by an approximate-identity limit; the cited results are used only for
the exact equality cases at zero smoothing time.

The section problem for $p>2$ is qualitatively different: the
identity of the extremizers can depend on the parameters, and even
specific non-Euclidean ranges require separate arguments; see
\cite{EskenazisNayarTkoczResilience}. At the Euclidean endpoint,
rotational invariance makes all equal-norm projections identically
distributed, so the assertion of
Theorem~\ref{thm:convex-order-main} reduces to equality. The cutoff at $p=2$ is sharp. To see this, assume $n\ge2$, let
$p>2$, put $q=p/(p-1)<2$, and take
\[
 a=e_1,
 \qquad
 b=\frac{e_1+e_2}{\sqrt2}.
\]
Then $\sqvec a\succ\sqvec b$, but
\[
 \sup_{x\in B_p^n}|\langle a,x\rangle|=1,
 \qquad
 \sup_{x\in B_p^n}|\langle b,x\rangle|
 =\|b\|_q=2^{1/q-1/2}>1.
\]
Hence
$\E(\langle b,X\rangle^2-1)_+>0$ while
$\E(\langle a,X\rangle^2-1)_+=0$, contradicting the conclusion of
Theorem~\ref{thm:convex-order-main}. Thus a theory for $p>2$ must use a
different stochastic order or a parameter-dependent comparison.

The comparison is also stable under a common independent radial scaling.

\begin{corollary}[Scale mixtures and radially decreasing laws]\label{cor:radial-intro}
Let $n\ge1$ and $1\le p\le2$, let $X$ be uniform on $B_p^n$, and let
$R\ge0$ be
independent of $X$ with $\E R^2<\infty$. If $\sqvec a\succ\sqvec b$, then
\[
 \langle b,RX\rangle^2\cx\langle a,RX\rangle^2.
\]
Consequently, the same conclusion holds for every probability measure
on $\R^n$ with finite second moment and density
\[
 x\longmapsto \rho(\|x\|_p),
\]
where $\rho:[0,\infty)\to[0,\infty)$ is nonincreasing.
\end{corollary}

The proof of Theorem~\ref{thm:convex-order-main} uses two ingredients.
The first is a direct two-dimensional comparison under a balancing of
two coefficients. The second is a preservation principle that allows
common coordinates to be appended without destroying squared convex
order. This preservation result is stated independently because it is
the main probabilistic mechanism of the proof.

\begin{theorem}[Beta--Rademacher preservation]\label{thm:app-append}
Let $1/2\le\alpha<1$, let $\beta\ge1+2\alpha$, and let $c\ge0$. Let
$Y_1,Y_2$ be square-integrable real random variables. Let
$T\sim\operatorname{Beta}(\alpha,\beta)$ and let $\eps$ be a Rademacher
sign, with $(T,\eps)$ independent of $(Y_1,Y_2)$. If
\[
 Y_1^2\cx Y_2^2,
\]
then
\[
 \left((1-T)^\alpha Y_1+cT^\alpha\eps\right)^2
 \cx
 \left((1-T)^\alpha Y_2+cT^\alpha\eps\right)^2.
\]
\end{theorem}

For $1<p<2$, put $\alpha=1/p$. The sign--Dirichlet representation of
the uniform measure on $B_p^n$ shows that adjoining a common coordinate
in dimension $d+1$ has precisely the form in
Theorem~\ref{thm:app-append}, with
\[
 \beta=1+d\alpha\ge1+2\alpha.
\]
Since majorization is generated by elementary two-coordinate
balancing operations, the planar comparison and the preservation
theorem imply Theorem~\ref{thm:convex-order-main} for $1<p<2$. The
endpoint $p=1$ follows by weak convergence, while $p=2$ follows directly
from rotational invariance.

The paper is organized as follows. Section~\ref{sec:preliminaries}
collects the majorization, convex-order, and Dirichlet facts used in the
proof. Section~\ref{sec:two-dimensional} establishes the planar
comparison. Section~\ref{sec:preservation} proves
Theorem~\ref{thm:app-append}. Section~\ref{sec:main-proof} completes the proof of
Theorem~\ref{thm:convex-order-main}, including strictness for $p<2$
and the endpoint arguments at $p=1,2$. Section~\ref{sec:extensions} proves the radial
extension and records moment and Laplace-transform consequences.
Section~\ref{sec:applications} treats Gaussian convolution and central
hyperplane sections.

\section{Majorization, convex order, and Dirichlet laws}\label{sec:preliminaries}

\subsection{Majorization and elementary transfers}

For $u\in\R^n$, let $u^\downarrow$ denote the vector obtained by
rearranging its coordinates in nonincreasing order. For vectors
$u,v\in\R^n$ with the same coordinate sum, one has $u\succ v$ if
\[
 \sum_{i=1}^k u_i^\downarrow
 \ge
 \sum_{i=1}^k v_i^\downarrow,
 \qquad 1\le k<n.
\]
An elementary $T$-transform replaces two coordinates $(x,y)$ by a more
balanced pair with the same sum and leaves all remaining coordinates
fixed. The Hardy--Littlewood--P\'olya characterization of majorization
states that $u\succ v$ if and only if $v$ can be obtained from $u$ by
finitely many such transforms, up to a permutation
\cite[Ch.~2]{MarshallOlkinArnold}. We shall apply this to the
nonnegative vectors $\sqvec a$ and $\sqvec b$.

\subsection{Convex order and hinge functions}

For nonnegative integrable random variables, it is enough to test
convex order on affine functions and one-parameter families of hinge
functions. We use the lower-hinge form because it is adapted to the
two-dimensional boundary calculation below. The following
characterization is standard; see, for example,
\cite[Ch.~3]{ShakedShanthikumar}.

\begin{lemma}[Hinge-function characterization of convex order]\label{lem:app-stoploss}
Let $U,V$ be nonnegative integrable random variables. Then $U\cx V$ if and only if $\E U=\E V$ and
\[
 \E(a-U)_+\le \E(a-V)_+\qquad(a\ge0).
\]
\end{lemma}

\begin{proof}
The usual stop-loss characterization for integrable random variables says that $U\cx V$ is equivalent to equality of means and
\[
 \E(U-a)_+\le\E(V-a)_+\qquad(a\ge0).
\]
All expectations are finite because $U,V$ are nonnegative and integrable. Since $(a-U)_+=a-U+(U-a)_+$, equality of means makes the upper and lower stop-loss forms equivalent.
\end{proof}

\subsection{The sign--Dirichlet model}

For the rest of this section assume $1<p<2$ and set
\[
 \alpha=\frac1p\in\left(\frac12,1\right).
\]
The following representation is standard
\cite[Section~2]{BartheGuedonMendelsonNaor}.
\begin{lemma}[Sign--Dirichlet representation]\label{lem:app-dirichlet}
Let $X^{(n)}$ be uniform on $B_p^n$. Then
\[
 X_i=\eps_iT_i^{\alpha},\qquad i=1,\ldots,n,
\]
where $\eps_1,\ldots,\eps_n$ are independent Rademacher signs and
\[
 (T_0,T_1,\ldots,T_n)\sim\operatorname{Dirichlet}(1,\alpha,\ldots,\alpha),
\]
independently of the signs. Here $T_0=1-\sum_{i=1}^n |X_i|^p$.
\end{lemma}

\begin{proof}
In the positive orthant set $t_i=x_i^p$. Then
\[
 \dd x_1\cdots\dd x_n=\alpha^n\prod_{i=1}^n t_i^{\alpha-1}\,\dd t_1\cdots\dd t_n,
\]
with $t_i\ge0$ and $t_1+\cdots+t_n\le1$. Adding the slack coordinate $t_0=1-\sum_i t_i$ gives the stated Dirichlet law. The signs are independent and uniform by unconditionality.
\end{proof}

\begin{lemma}[Dirichlet append step]\label{lem:app-append-distribution}
Let $Y$ be a projection generated by a $d$-dimensional $p$-ball through the sign-Dirichlet representation. Appending one coordinate with coefficient $c\ge0$ gives
\begin{equation}\label{eq:app-append}
 \widetilde Y=(1-T)^\alpha Y+cT^\alpha\eps,
\end{equation}
where $\eps$ is a Rademacher sign and
\[
 T\sim\operatorname{Beta}(\alpha,1+d\alpha)
\]
is independent of $Y$ and $\eps$.
\end{lemma}

\begin{proof}
In the $(d+1)$-dimensional Dirichlet vector, aggregate the first $d$ non-slack masses and the slack into one block. The last coordinate mass $T$ has beta law $\operatorname{Beta}(\alpha,1+d\alpha)$. Conditional on $T$, the remaining normalized masses have the $d$-dimensional sign-Dirichlet law. This gives \eqref{eq:app-append}.
\end{proof}

\section{The two-dimensional comparison}\label{sec:two-dimensional}

Throughout this section let $1<p<2$ and put $\alpha=1/p\in(1/2,1)$.
The goal is to compare squared projections when two coefficients are
balanced while their Euclidean norm is held fixed. The proof is made
at the level of lower hinge functions and uses a boundary
differentiation on the planar ball $B_p^2$.

\begin{lemma}[Two-interval comparison]\label{lem:app-two-interval}
Fix $\ell\ge0$, $0\le B\le1$, and $0\le x\le y$. Define
\[
 \Phi_\ell(t)=(\ell^2-t^2)_+,
 \qquad
 Q_\ell(x,y)=\Phi_\ell(By-x)-\Phi_\ell(By+x).
\]
Then
\[
 Q_\ell(x,y)\ge Q_\ell(y,x).
\]
\end{lemma}

\begin{proof}
Since $\Phi_\ell$ is even,
\[
 Q_\ell(x,y)=\Phi_\ell(|By-x|)-\Phi_\ell(By+x).
\]
For $A\le C$,
\[
 \Phi_\ell(A)-\Phi_\ell(C)=\int_A^C2r\1_{\{r<\ell\}}\,\dd r.
\]
Thus $Q_\ell(x,y)$ is the weighted mass of the interval
\[
 I_1=[|By-x|,By+x],
\]
while $Q_\ell(y,x)$ is the corresponding mass of
\[
 I_2=[y-Bx,y+Bx].
\]
Write
\[
 a_1=|By-x|,\quad b_1=By+x,\qquad
 a_2=y-Bx,\quad b_2=y+Bx.
\]
Then $a_1\le a_2$, $b_1\le b_2$, and
\[
 b_1^2-a_1^2=b_2^2-a_2^2=4Bxy.
\]
For any interval $[a,b]$ with $a\le b$,
\[
 \int_{[a,b]\cap[0,\ell]}2r\,\dd r
 =
 \begin{cases}
 0, & \ell\le a,\\
 \ell^2-a^2, & a<\ell<b,\\
 b^2-a^2, & b\le\ell.
 \end{cases}
\]
We compare the two intervals by the possible positions of $\ell$. If $\ell\le a_1$, both masses are zero. If $a_1<\ell\le a_2$, only the first mass is positive. If $a_2<\ell\le b_1$, the difference is
\[
 (\ell^2-a_1^2)-(\ell^2-a_2^2)=a_2^2-a_1^2\ge0.
\]
If $b_1<\ell\le b_2$, the first mass is $b_1^2-a_1^2=b_2^2-a_2^2$, while the second is at most this value. If $b_2<\ell$, the two masses are equal. Thus the mass of $I_1$ is always at least the mass of $I_2$, proving the claim.
\end{proof}

\begin{lemma}[Boundary identity in dimension two]\label{lem:app-base-boundary}
Let $\rho\ge0$ and $u(s)=(1-s^p)^\alpha$. Let
\[
 \theta(\varphi)=ae_1+be_2,
 \qquad
 a=\sqrt\sigma\cos\varphi,
 \qquad
 b=\sqrt\sigma\sin\varphi,
\]
with $a>0$ and $0\le b\le a$. Put $B=b/a$ and $\ell=\rho/a$. If
\[
 F_\rho(\varphi)=\int_{B_p^2}(\rho^2-\langle\theta(\varphi),x\rangle^2)_+\,\dd x,
\]
then
\begin{align}\label{eq:app-base-boundary}
 -F_\rho'(\varphi)
 &=2a^2\alpha\int_0^{1/2}
 \big((1-t)^{2\alpha-1}-t^{2\alpha-1}\big) \notag\\
 &\qquad\times
 \big[Q_\ell(t^\alpha,(1-t)^\alpha)-Q_\ell((1-t)^\alpha,t^\alpha)\big] \dd t.
\end{align}
\end{lemma}

\begin{proof}
Let $V(x_1,x_2)=(x_2,-x_1)$. Then $\operatorname{div}V=0$ and
\[
 V\cdot\nabla\langle\theta(\varphi),x\rangle=\langle\theta'(\varphi),x\rangle.
\]
A standard Lipschitz approximation justifies applying the divergence theorem to $(\rho^2-\langle\theta,x\rangle^2)_+$. On the positive boundary arc $x=(s,u(s))$, the outward normal measure is $(-u'(s),1)\dd s$, so
\[
 V\cdot n\,\dd\sigma=(-u(s)u'(s)-s)\dd s.
\]
Summing over the four signed arcs and using evenness gives
\[
 -F_\rho'(\varphi)
 =2a^2\int_0^1(-uu'-s)
 \left[\Phi_\ell(Bu-s)-\Phi_\ell(Bu+s)\right]\dd s.
\]
Now set $t=s^p$. Then $s=t^\alpha$, $u=(1-t)^\alpha$, and
\[
 (-uu'-s)\dd s
 =\alpha\big((1-t)^{2\alpha-1}-t^{2\alpha-1}\big)\dd t.
\]
Pairing $t$ with $1-t$ gives \eqref{eq:app-base-boundary}.
\end{proof}

\begin{proposition}[Two-dimensional hinge comparison]\label{prop:app-base2}
Let $X^{(2)}$ be uniform on $B_p^2$, $1<p<2$. If $\theta,\eta\in\R^2$ have the same Euclidean norm and
\[
 (\theta_1^2,\theta_2^2)\succ(\eta_1^2,\eta_2^2),
\]
then
\[
 \langle\eta,X^{(2)}\rangle^2\cx\langle\theta,X^{(2)}\rangle^2.
\]
\end{proposition}

\begin{proof}
By unconditionality and permutation symmetry, it is enough to consider
\[
 \theta(\varphi)=ae_1+be_2,
 \qquad
 a=\sqrt\sigma\cos\varphi,
 \qquad
 b=\sqrt\sigma\sin\varphi,
 \qquad
 0\le\varphi\le\pi/4.
\]
Increasing $\varphi$ balances the squared coordinates. For $\rho>0$ define
\[
 F_\rho(\varphi)=\int_{B_p^2}(\rho^2-\langle\theta(\varphi),x\rangle^2)_+\,\dd x.
\]
We show that $F_\rho$ is nonincreasing in $\varphi$. If $a=b$ or $b=0$, this follows by limiting. Otherwise Lemma~\ref{lem:app-base-boundary} applies. For $0\le t\le1/2$, set $x=t^\alpha$ and $y=(1-t)^\alpha$. Then $0\le x\le y$, so Lemma~\ref{lem:app-two-interval} gives
\[
 Q_\ell(x,y)\ge Q_\ell(y,x).
\]
Also $(1-t)^{2\alpha-1}-t^{2\alpha-1}\ge0$ because $2\alpha-1>0$. Hence $-F_\rho'(\varphi)\ge0$.

The means are equal along the path, since
\[
 \E\langle\theta(\varphi),X^{(2)}\rangle^2
 =\E\bigl[(X_1^{(2)})^2\bigr](a^2+b^2).
\]
Lemma~\ref{lem:app-stoploss} gives the desired convex order.
\end{proof}
\section{Beta--Rademacher preservation}\label{sec:preservation}

Throughout this section let $1/2<\alpha<1$ and set $p=1/\alpha$.
We prove Theorem~\ref{thm:app-append}. For a hinge level $\rho>0$ and
$q\ge0$, condition on an input satisfying $Y^2=q$ and define the
conditional lower-hinge transform
\[
 G(q)=\E_{T,\eps}\left(
 \rho^2-\left((1-T)^\alpha\sqrt q+cT^\alpha\eps\right)^2
 \right)_+.
\]
The preservation theorem follows once $G$ is shown to be convex. A
distributional second-derivative formula reduces this convexity to a
weighted layer inequality. We first record that formula, then establish
the required centered and dual layer inequalities.

\subsection{A distributional second-derivative formula}

\begin{lemma}[Quadratic layer second derivative]\label{lem:app-second-dist}
Let $L,D,W$ be $C^1$ functions on an interval
$I\subset(0,\infty)$, with $L>0$, $D\ge0$, and $W\ge0$. Assume that
the measures and integrals below are locally finite. Define, for $x>0$,
\[
 F(x)=\int_I W(R)\frac12\Big(
 (L(R)^2-(x+D(R))^2)_+
 +(L(R)^2-(x-D(R))^2)_+
 \Big)\,\dd R,
\]
and set $G(q)=F(\sqrt q)$ for $q>0$. Let
\[
 \nu_-=(|L-D|)_\#(W L\,\dd R),
 \qquad
 \nu_+=(L+D)_\#(W L\,\dd R)
\]
be the pushforward measures on $(0,\infty)$. Then, as a signed
distribution in the variable $x>0$,
\begin{equation}\label{eq:app-dist-measure}
 xF''-F'
 =
 x(\nu_-+\nu_+)
 -
 \left(
 \int_{\{|L-D|<x<L+D\}}W(R)D(R)\,\dd R
 \right)\dd x.
\end{equation}
Consequently, at every common regular value $x>0$ of $|L-D|$ and
$L+D$,
\begin{align}\label{eq:app-dist-lemma}
4x^3G''(x^2)=&\
 x\sum_{|L-D|=x}\frac{W(R)L(R)}{|(L-D)'(R)|}
 +x\sum_{L+D=x}\frac{W(R)L(R)}{|(L+D)'(R)|}\notag\\
&\quad
 -\int_{\{|L-D|<x<L+D\}}W(R)D(R)\,\dd R.
\end{align}
In particular, critical or colliding roots are already accounted for
by the pushforward measures in \eqref{eq:app-dist-measure}; no
separate choice of root branches is required.
\end{lemma}

\begin{proof}
For fixed $a\in\R$ and $L>0$, put
\[
 \phi_{L,a}(x)=(L^2-(x-a)^2)_+.
\]
As distributions in $x$,
\[
 \phi_{L,a}'=-2(x-a)\1_{\{|x-a|<L\}},
\]
and
\[
 \phi_{L,a}''
 =
 -2\1_{\{|x-a|<L\}}
 +2L\delta_{a-L}
 +2L\delta_{a+L}.
\]
It follows that
\[
 x\phi_{L,a}''-\phi_{L,a}'
 =
 -2a\1_{\{|x-a|<L\}}
 +2xL(\delta_{a-L}+\delta_{a+L}).
\]
Apply this identity with $a=D(R)$ and $a=-D(R)$, average, integrate
against $W(R)\dd R$, and use Fubini's theorem for distributions. On
the positive half-line, the endpoint Dirac masses push forward to
$|L-D|$ and $L+D$, while the difference of the two indicator terms is
\[
 -D(R)\1_{\{|L(R)-D(R)|<x<L(R)+D(R)\}}.
\]
This proves \eqref{eq:app-dist-measure}.

At a regular value, the one-dimensional coarea formula gives
\[
 \frac{\dd\nu_-}{\dd x}
 =
 \sum_{|L-D|=x}\frac{W(R)L(R)}{|(L-D)'(R)|},
 \qquad
 \frac{\dd\nu_+}{\dd x}
 =
 \sum_{L+D=x}\frac{W(R)L(R)}{|(L+D)'(R)|}.
\]
Finally, the change of variables $q=x^2$ gives, in the smooth case,
\[
 G''(x^2)=\frac{xF''(x)-F'(x)}{4x^3},
\]
and the same identity holds distributionally by testing against
compactly supported smooth functions. This yields
\eqref{eq:app-dist-lemma}.
\end{proof}

\subsection{The centered layer inequality}

The next estimate compares the endpoint contribution of a one-dimensional
layer with its weighted bulk. The exponent condition
$\delta\ge3\alpha$ is exactly the condition that will arise from
$\beta\ge1+2\alpha$ in the preservation theorem. The component
estimates use the following weighted trapezoid inequality.

\begin{lemma}[Weighted beta-trapezoid]\label{lem:app-beta-trap}
Let $1/2<\alpha<1$ and let $\delta\ge3\alpha$. Then for $0\le a<b\le1$,
\begin{equation}\label{eq:app-beta-trap}
 \alpha\int_a^b t^{2\alpha-1}(1-t)^\delta\,\dd t
 \le
 \frac{b^{2\alpha}-a^{2\alpha}}4\left((1-a)^\delta+(1-b)^\delta\right).
\end{equation}
\end{lemma}

\begin{proof}
Put $y=t^{2\alpha}$ and $\mu=1/(2\alpha)\in(1/2,1)$. The integral on the left of \eqref{eq:app-beta-trap} is
\[
 \frac12\int_{a^{2\alpha}}^{b^{2\alpha}}(1-y^\mu)^\delta\,\dd y.
\]
For $F(y)=(1-y^\mu)^\delta$,
\[
 F''(y)=\delta\mu y^{\mu-2}(1-y^\mu)^{\delta-2}\bigl(1-\mu+(\mu\delta-1)y^\mu\bigr).
\]
Since $\delta\ge3\alpha$, $\mu\delta\ge3/2>1$, so $F$ is convex. The trapezoid rule for convex functions gives the claim.
\end{proof}

For $B>0$ set
\[
 u(s)=(1-s^p)^\alpha,
 \qquad
 H(s)=Bu(s)+s,
 \qquad
 G(s)=Bu(s)-s,
 \qquad 0\le s\le1.
\]

\begin{lemma}[Root configurations]\label{lem:app-root-config}
The function $G$ is strictly decreasing on $[0,1]$. The function $H$ has at most one critical point; it is increasing before this point and decreasing after it. Consequently, for a fixed $L>0$, every connected component of
\[
 \mathcal A_L=\{s\in[0,1]: |G(s)|<L<H(s)\}
\]
is central, mixed, or cap, in the following sense.
\begin{enumerate}[label=(\roman*),leftmargin=2em]
\item If $0<L<1$, then either the component is central, opened by $G=L$ and closed by $G=-L$, or it is mixed, opened by $H=L$ and closed by $G=-L$.
\item If $L>1$, then either it is opened by an increasing-branch root of $H=L$ and closed by a decreasing-branch root of $H=L$, or it is opened by $G=L$ and closed by a decreasing-branch root of $H=L$.
\item The cases $L=1$, endpoint roots, and tangencies are limits of the preceding cases.
\end{enumerate}
\end{lemma}

\begin{proof}
Differentiating gives
\[
 u'(s)=-s^{p-1}(1-s^p)^{\alpha-1}.
\]
Thus $G'(s)=Bu'(s)-1<0$. Also
\[
 H'(s)=1-Bs^{p-1}(1-s^p)^{\alpha-1}.
\]
The positive function $s\mapsto s^{p-1}(1-s^p)^{\alpha-1}$ is strictly increasing from $0$ to $+\infty$, so $H'$ changes sign at most once. The configurations follow by intersecting $-L<G<L$ and $H>L$ with these monotonicity facts.
\end{proof}

\begin{theorem}[Weighted centered layer theorem]\label{thm:app-wcl}
Let $1/2<\alpha<1$, $\delta\ge3\alpha$, $B>0$, and $L>0$.
Assume that all roots in $(0,1)$ of $|G|=L$ and $H=L$ are simple and
that no such root lies at an endpoint. Then
\begin{equation}\label{eq:app-wcl}
 LB\sum_{|G|=L}\frac{(1-s^p)^{\delta+\alpha}}{|G'(s)|}
 +LB\sum_{H=L}\frac{(1-s^p)^{\delta+\alpha}}{|H'(s)|}
 \ge
 \int_{\{|G|<L<H\}}s(1-s^p)^\delta\,\dd s.
\end{equation}
\end{theorem}

We prove the theorem component by component. In all component calculations below we use the variable
\[
 t=s^p,
 \qquad s=t^\alpha,
 \qquad u(s)=(1-t)^\alpha.
\]
Thus letters $a,b$ denote $t$-coordinates of the opening and closing roots. At a root of $H=L$, respectively of $G=\pm L$, the coarea denominators in the original $s$-variable become
\begin{align}\label{eq:app-endpoint-denominators}
 \frac{(1-s^p)^{\delta+\alpha}}{|H'(s)|}
 &=\frac{(1-t)^{\delta+\alpha+1}}{|1-Lt^{1-\alpha}|},\notag\\
 \frac{(1-s^p)^{\delta+\alpha}}{|G'(s)|}
 &=\frac{(1-t)^{\delta+\alpha+1}}{|1\pm Lt^{1-\alpha}|},
\end{align}
where the sign is $+$ at a root of $G=L$ and $-$ at a root of $G=-L$. For example, if $H=L$, then
\[
 B(1-t)^\alpha+t^\alpha=L,
\]
so
\[
 H'(s)=1-Bt^{1-\alpha}(1-t)^{\alpha-1}
 =\frac{1-Lt^{1-\alpha}}{1-t}.
\]
The two formulas for $G$ are identical, using $B(1-t)^\alpha-t^\alpha=\pm L$.

\begin{lemma}[Universal polynomial inequality]\label{lem:app-poly}
Let $0<s<1$ and $s^2\le r\le s$. Then
\begin{equation}\label{eq:app-poly-ineq}
 (1+s^2)(1+r)(1+r^3)\le2(1+r^2s)^2.
\end{equation}
\end{lemma}

\begin{proof}
Let
\[
 F_s(r)=2(1+r^2s)^2-(1+s^2)(1+r)(1+r^3).
\]
Then
\[
 F_s(s)=(1-s)^2(1+s)^2(s^2-s+1)\ge0.
\]
Write $r=s^2+z(s-s^2)$, $0\le z\le1$. Expanding in powers of $z$ gives
\[
 F_s(r)-F_s(s)=(r-s)H_s(r).
\]
Since $r-s\le0$, it is enough to prove $H_s(r)\le0$. In Bernstein form as a polynomial in $z$, the polynomial
\[
 -H_s(s^2+z(s-s^2))
\]
has coefficients
\[
 (1-s)P_0(s),\quad \frac{(1-s)(1+s)}3P_1(s),\quad
 \frac{1-s}{3}P_2(s),\quad (1-s)(1+s)P_3(s),
\]
where
\begin{align*}
P_0(s)&=s^7+2s^6+s^5-3s^3-s^2+s+1,\\
P_1(s)&=4s^5+3s^4-6s^2+3,\\
P_2(s)&=6s^5+9s^4-3s^3-8s^2+3s+3,\\
P_3(s)&=4s^3-3s^2+1.
\end{align*}
These four polynomials have nonnegative Bernstein coefficients on $[0,1]$:
\begin{align*}
P_0 &: 1,\frac87,\frac{26}{21},\frac65,\frac{33}{35},\frac37,0,2,\\
P_1 &: 3,3,\frac{12}{5},\frac65,0,4,\\
P_2 &: 3,\frac{18}{5},\frac{17}{5},\frac{21}{10},\frac65,10,\\
P_3 &: 1,1,0,2.
\end{align*}
Hence $-H_s\ge0$, so $F_s(r)\ge F_s(s)\ge0$.
\end{proof}

\begin{lemma}[Mixed component at the base exponent]\label{lem:app-mixed-component}
The estimate \eqref{eq:app-wcl} holds on every mixed component when $\delta=3\alpha$.
\end{lemma}

\begin{proof}
The mixed roots are described in $t$-coordinates by
\[
 B(1-a)^\alpha+a^\alpha=L,
 \qquad
 b^\alpha-B(1-b)^\alpha=L,
 \qquad 0<a<b<1.
\]
Set
\[
 X=a^\alpha,
 \quad Y=b^\alpha,
 \quad U=1-a,
 \quad V=1-b=wU,
 \quad r=w^\alpha.
\]
Solving the two root equations gives
\begin{equation}\label{eq:app-mixed-BL}
 B=\frac{Y-X}{U^\alpha+V^\alpha},
 \qquad
 L=\frac{YU^\alpha+XV^\alpha}{U^\alpha+V^\alpha}=\frac{Y+rX}{1+r}.
\end{equation}
The endpoint denominators from \eqref{eq:app-endpoint-denominators} are
\[
 d_a^+=1-La^{1-\alpha},
 \qquad d_b^+=1-Lb^{1-\alpha}.
\]
For $\delta=3\alpha$, the endpoint side on this component equals
\[
 LB\left(\frac{U^{4\alpha+1}}{d_a^+}
 +\frac{V^{4\alpha+1}}{d_b^+}\right)
 =\frac{L(Y-X)U^{3\alpha+1}}{1+r}
 \left(\frac1{d_a^+}+\frac{w^{4\alpha+1}}{d_b^+}\right).
\]
The beta-trapezoid estimate gives the bulk bound
\[
 \alpha\int_a^b t^{2\alpha-1}(1-t)^{3\alpha}\,\dd t
 \le
 \frac{Y^2-X^2}{4}U^{3\alpha}(1+w^{3\alpha}).
\]
After cancelling the common positive factor $(Y-X)U^{3\alpha}/(1+r)$, it remains to prove
\begin{equation}\label{eq:app-KM-base}
 \frac{4LU}{X+Y}\left(\frac1{d_a^+}+\frac{w^{4\alpha+1}}{d_b^+}\right)
 \ge(1+w^\alpha)(1+w^{3\alpha}).
\end{equation}
Cauchy's inequality gives
\[
 \frac1{d_a^+}+\frac{w^{4\alpha+1}}{d_b^+}
 \ge\frac{(1+w^{2\alpha+1/2})^2}{d_a^++d_b^+}.
\]
We claim that
\begin{equation}\label{eq:app-mixed-compress}
 \frac{(d_a^++d_b^+)(X+Y)}{U(Y+rX)}\le\frac{2(1+w)}{1+r}.
\end{equation}
To verify it, put $c=a/b$, $u_0=c^{1-\alpha}$, and $k=\alpha/(1-\alpha)>1$. Then $c^\alpha=u_0^k$, and \eqref{eq:app-mixed-compress} reduces to
\[
 (1-w)(r-u_0)(1+u_0^k)
 \le (1-u_0^{k+1})(1+w)(1-r).
\]
This follows from the three inequalities
\[
 r-u_0\le r(1-u_0),
 \qquad
 1-u_0^{k+1}\ge(1-u_0)(1+u_0^k),
 \qquad
 r(1-w)\le(1+w)(1-r).
\]
The last inequality is equivalent to $2r\le1+w$. Since $\alpha>1/2$, we have $r=w^\alpha\le\sqrt w$, and hence
\[
 2r\le2\sqrt w\le1+w.
\]
Combining Cauchy with \eqref{eq:app-mixed-compress}, \eqref{eq:app-KM-base} follows from Lemma~\ref{lem:app-poly} with $s=\sqrt w$.
\end{proof}

\begin{lemma}[Central component at the base exponent]\label{lem:app-central-component}
The estimate \eqref{eq:app-wcl} holds on every central component when $\delta=3\alpha$.
\end{lemma}

\begin{proof}
The central roots are described by
\[
 B(1-a)^\alpha-a^\alpha=L,
 \qquad
 b^\alpha-B(1-b)^\alpha=L,
 \qquad 0<a<b<1.
\]
With the same notation $X,Y,U,V,w,r$ as in the mixed case,
\begin{equation}\label{eq:app-central-BL}
 B=\frac{X+Y}{U^\alpha+V^\alpha},
 \qquad
 L=\frac{YU^\alpha-XV^\alpha}{U^\alpha+V^\alpha}=\frac{Y-rX}{1+r}.
\end{equation}
The endpoint denominators are
\[
 d_a^-=1+La^{1-\alpha},
 \qquad d_b^-=1-Lb^{1-\alpha}.
\]
Exactly as above, the beta-trapezoid estimate reduces the desired inequality to
\begin{equation}\label{eq:app-KC-base}
 \frac{4LU}{Y-X}\left(\frac1{d_a^-}+\frac{w^{4\alpha+1}}{d_b^-}\right)
 \ge(1+w^\alpha)(1+w^{3\alpha}).
\end{equation}
After Cauchy's inequality and Lemma~\ref{lem:app-poly}, it remains to prove the compression estimate
\begin{equation}\label{eq:app-central-compress}
 \frac{(d_a^-+d_b^-)(Y-X)}{U(Y-rX)}\le\frac{2(1+w)}{1+r}.
\end{equation}
Put $c=a/b$, $u_0=c^{1-\alpha}$, $q=u_0^k$, and $k=\alpha/(1-\alpha)>1$. Since $c=qu_0$, the elementary relations
\[
 U=\frac{1-qu_0}{1-qu_0w},
 \qquad
 b=\frac{1-w}{1-qu_0w},
 \qquad
 X=qY
\]
give the exact identity
\begin{align}\label{eq:central-positive-factor}
&\frac{2(1+w)}{1+r}
-\frac{(d_a^-+d_b^-)(Y-X)}{U(Y-rX)} \notag\\
&\quad =
\frac{1+q}{(1+r)(1-rq)(1-qu_0)}\notag\\
&\qquad\times
\Big(qru_0w+qru_0-qrw+qr-2qu_0w-2r+u_0w-u_0+w+1\Big).
\end{align}
The prefactor in \eqref{eq:central-positive-factor} is positive because $0<q,r,u_0<1$. It remains to show
\[
 B_0(u):=qruw+qru-qrw+qr-2quw-2r+uw-u+w+1\ge0
\]
on $0\le u\le1$, where $q=u^k$ and $r=w^\alpha$.

Write $w=rt$ with $0<t<1$ and $r=t^k$. Then $q=u^k$. Differentiating gives
\[
 B_0'(u)=w-1+(k+1)Au^k+kr(1-w)u^{k-1},
 \qquad A=r(1+w)-2w.
\]
First,
\[
 B_0'(1)=-1+(2k+1)t^k(1-t)+t^{2k+1}\le0.
\]
Indeed,
\[
 \Phi(t)=(2k+1)t^k(1-t)+t^{2k+1}
\]
is increasing on $(0,1)$ and equals $1$ at $t=1$, since
\[
 \Phi'(t)=(2k+1)t^{k-1}\bigl(k-(k+1)t+t^{k+1}\bigr)>0.
\]
Let $F=B_0'$. If $A\ge0$, then $F$ is increasing, so $F(u)\le F(1)\le0$ for $u\le1$. If $A<0$, then
\[
 F'(u)=ku^{k-2}\bigl((k+1)Au+(k-1)r(1-w)\bigr)
\]
has at most one zero. Since $F(0)=w-1<0$, a nonnegative maximum of $F$ can occur only at $u=1$, already handled, or at an interior critical point $u_*$ satisfying
\[
 (k+1)Au_*+(k-1)r(1-w)=0.
\]
At such a point,
\[
 F(u_*)=w-1+r(1-w)u_*^{k-1}<0.
\]
Thus $B_0'(u)\le0$ on $[0,1]$. Since $B_0(1)=0$, we obtain $B_0(u)\ge0$. This proves \eqref{eq:app-central-compress}, and hence \eqref{eq:app-KC-base}.
\end{proof}

\begin{lemma}[Cap components]\label{lem:app-cap-component}
The estimate \eqref{eq:app-wcl} holds on every cap component for every $\delta\ge3\alpha$.
\end{lemma}

\begin{proof}
Assume $L>1$. If the active component is empty, there is nothing to prove. Suppose first that the component begins at a root of $H=L$ on the increasing branch, written in $t$-coordinates as
\[
 B(1-a)^\alpha+a^\alpha=L.
\]
The opening endpoint contribution is
\[
 L(L-a^\alpha)\frac{(1-a)^{\delta+1}}{1-La^{1-\alpha}}.
\]
The entire bulk is bounded by
\[
 \alpha\int_a^1 t^{2\alpha-1}(1-t)^\delta\,\dd t
 \le \alpha\int_a^1(1-t)^\delta\,\dd t
 =\frac{\alpha}{\delta+1}(1-a)^{\delta+1}.
\]
Thus it suffices to show
\[
 \frac{L(L-a^\alpha)}{1-La^{1-\alpha}}\ge\frac{\alpha}{\delta+1}.
\]
The left side is at least $1$: $L(L-a^\alpha)\ge1-a^\alpha$, while $1-La^{1-\alpha}\le1-a^{1-\alpha}$ and $a^\alpha\le a^{1-\alpha}$. This proves the first cap configuration.

If the component begins at a root of $G=L$,
\[
 B(1-a)^\alpha-a^\alpha=L,
\]
the opening endpoint contributes
\[
 L(L+a^\alpha)\frac{(1-a)^{\delta+1}}{1+La^{1-\alpha}}.
\]
For fixed $a$, the function
\[
 L\mapsto \frac{L(L+a^\alpha)}{1+La^{1-\alpha}}
\]
is increasing for $L\ge1$, because its derivative is
\[
 \frac{2L+a^\alpha+L^2a^{1-\alpha}}{(1+La^{1-\alpha})^2}>0.
\]
Therefore it is at least $(1+a^\alpha)/(1+a^{1-\alpha})\ge1/2$, whereas $\alpha/(\delta+1)\le\alpha/(3\alpha+1)<1/2$. Endpoint and tangency cases are obtained by monotone limiting of the component endpoints and dominated convergence of the bulk.
\end{proof}

\begin{lemma}[Extension from $\delta=3\alpha$ to $\delta\ge3\alpha$]\label{lem:app-delta-extension}
If \eqref{eq:app-wcl} holds on all central and mixed components for $\delta=3\alpha$, then it holds on those components for every $\delta\ge3\alpha$.
\end{lemma}

\begin{proof}
In the mixed range the same reduction as in Lemma~\ref{lem:app-mixed-component} gives
\[
 \frac{4LU}{X+Y}\left(\frac1{d_a^+}+\frac{w^{\delta+\alpha+1}}{d_b^+}\right)
 \ge (1+w^\alpha)(1+w^\delta).
\]
This expression is affine in $z=w^\delta$. Since $0\le z\le w^{3\alpha}$, it is enough to check $z=w^{3\alpha}$ and $z=0$. The first value is the base case. At $z=0$ we need
\[
 \frac{4LU}{(X+Y)d_a^+}\ge1+r.
\]
Because $L\ge X$, one has $d_a^+\le1-Xa^{1-\alpha}=U$. Thus it suffices to prove
\[
 4L\ge(1+r)(X+Y),
\]
which, using $L=(Y+rX)/(1+r)$, is equivalent to
\[
 (1-r)((3+r)Y-(1-r)X)\ge0.
\]

In the central range the reduction gives
\[
 \frac{4LU}{Y-X}\left(\frac1{d_a^-}+\frac{w^{\delta+\alpha+1}}{d_b^-}\right)
 \ge (1+w^\alpha)(1+w^\delta).
\]
Again this is affine in $z=w^\delta$. The endpoint $z=w^{3\alpha}$ is the base case. At $z=0$ it remains to prove
\[
 \frac{4LU}{(Y-X)d_a^-}\ge1+r.
\]
We use two estimates. First,
\[
 d_a^-\le\frac2{1+r}.
\]
Indeed, this follows from
\[
 a^{1-\alpha}Y\le1-rU,
\]
which is a consequence of weighted AM--GM:
\[
 a^{1-\alpha}b^\alpha\le(1-\alpha)a+\alpha b\le1-U+U(1-w^\alpha)=1-rU.
\]
Second,
\[
 2UL\ge Y-X.
\]
For this, write
\[
 Y-X=\alpha\int_a^b t^{\alpha-1}\,\dd t,
 \qquad
 Y-rX=\alpha\int_{wa}^{b} t^{\alpha-1}\,\dd t.
\]
Since $t^{\alpha-1}$ is decreasing and
\[
 \frac{|[wa,a]|}{|[a,b]|}=\frac{a-wa}{b-a}=\frac aU,
\]
the integral over $[wa,a]$ dominates $(a/U)$ times the integral over $[a,b]$. Therefore $Y-rX\ge(Y-X)/U$, and since $L=(Y-rX)/(1+r)$ with $1+r\le2$, we get $2UL\ge Y-X$. The two estimates give the desired central endpoint $z=0$.
\end{proof}

\begin{proof}[Proof of Theorem~\ref{thm:app-wcl}]
Under the regularity assumptions in the statement,
Lemma~\ref{lem:app-root-config} decomposes the active set into finitely
many central, mixed, and cap components. The endpoint sums and the bulk
integral are additive over these components.
Lemmas~\ref{lem:app-mixed-component} and~\ref{lem:app-central-component} prove the base exponent on mixed and
central components, Lemma~\ref{lem:app-cap-component} proves the cap
range for every allowed exponent, and
Lemma~\ref{lem:app-delta-extension} lifts the mixed and central
estimates to every $\delta\ge3\alpha$. Summing the component
inequalities proves \eqref{eq:app-wcl}.
\end{proof}

\subsection{The dual layer inequality}

The second-derivative formula produces a reciprocal version of the
centered layer geometry. The following change-of-variables statement
converts it back to Theorem~\ref{thm:app-wcl}.

\begin{theorem}[Dual layer theorem]\label{thm:app-dlt}
Let $1/2<\alpha<1$, let $\beta\ge1+2\alpha$, and let $A\ge0$. Set
\[
 h(s)=\frac{1+As}{(1-s^p)^\alpha},
 \qquad
 g(s)=\frac{1-As}{(1-s^p)^\alpha},
 \qquad 0\le s<1.
\]
Let $x>0$ be a common regular value of $|g|$ and $h$, and assume
that no corresponding root lies at an endpoint. Then
\begin{align}\label{eq:app-dlt}
 &x\sum_{|g|=x}\frac{(1-s^p)^{\beta+\alpha-1}}{|g'(s)|}
 +x\sum_{h=x}\frac{(1-s^p)^{\beta+\alpha-1}}{|h'(s)|} \\
 &\hspace{1in}\ge
 A\int_{\{|g|<x<h\}}s(1-s^p)^{\beta+\alpha-1}\,\dd s.
\end{align}
\end{theorem}

\begin{proof}
If $A=0$, the right side is zero. Otherwise put
\[
 u(s)=(1-s^p)^\alpha,
 \qquad
 B=\frac{x}{A},
 \qquad
 L=\frac1A.
\]
Then the endpoint equations for $g,h$ correspond to $G=\pm L$ and $H=L$ for the centered functions $G=Bu-s$ and $H=Bu+s$. Also
\[
 \{|g|<x<h\}=\{|G|<L<H\}.
\]
At corresponding endpoints,
\[
 |h'|=\frac{A}{u}|G'|,
 \qquad
 |g'|=\frac{A}{u}|H'|\ \text{on }g=x,
 \qquad
 |g'|=\frac{A}{u}|G'|\ \text{on }g=-x.
\]
Set $\delta=\beta+\alpha-1\ge3\alpha$. The endpoint side of \eqref{eq:app-dlt} is $A$ times the endpoint side of \eqref{eq:app-wcl}, and the right side is $A$ times the centered bulk. Theorem~\ref{thm:app-wcl} gives the claim.
\end{proof}
\subsection{Proof of the preservation theorem}

\begin{proof}[Proof of Theorem~\ref{thm:app-append}]
We first assume $1/2<\alpha<1$. It suffices to check lower-hinge functions $\Phi_\rho(q)=(\rho^2-q)_+$ and affine functions. Affine functions are preserved because the transformation adds the same second-moment contribution to both sides. The case $\rho=0$ is trivial, since $\Phi_0(q)\equiv0$ on $[0,\infty)$; hence assume $\rho>0$.

For $q\ge0$ set
\[
 G(q)=\E_{T,\eps}\left(\rho^2-\left((1-T)^\alpha\sqrt q+cT^\alpha\eps\right)^2\right)_+.
\]
The Rademacher average makes this depend only on $q$. We prove that $G$ is convex. Put
\[
 R=\left(\frac{T}{1-T}\right)^\alpha.
\]
Then
\[
 (1-T)^\alpha=(1+R^p)^{-\alpha},
 \qquad
 T^\alpha=R(1+R^p)^{-\alpha},
\]
and the density of $R$ is a positive multiple of $(1+R^p)^{-\alpha-\beta}\dd R$. Hence, up to a positive constant,
\[
G(q)=\int_0^\infty (1+R^p)^{-\beta-3\alpha}\frac12\sum_{\pm}
\left(\rho^2(1+R^p)^{2\alpha}-(\sqrt q\pm cR)^2\right)_+\,\dd R.
\]
Set
\[
 L(R)=\rho(1+R^p)^\alpha,
 \qquad
 D(R)=cR,
 \qquad
 W(R)=(1+R^p)^{-\beta-3\alpha}.
\]
Lemma~\ref{lem:app-second-dist} gives a second derivative formula. Now pass to
\[
 s=\frac{R}{(1+R^p)^{1/p}},
 \qquad
 R=\frac{s}{(1-s^p)^\alpha},
 \qquad
 \dd R=(1-s^p)^{-\alpha-1}\dd s.
\]
Then
\[
 1+R^p=(1-s^p)^{-1},
 \qquad
 L(R)\pm D(R)=\rho\frac{1\pm(c/\rho)s}{(1-s^p)^\alpha}.
\]
Writing $A=c/\rho$ and $y=x/\rho$, the dual layer functions are
\[
 h(s)=\frac{1+As}{(1-s^p)^\alpha},
 \qquad
 g(s)=\frac{1-As}{(1-s^p)^\alpha}.
\]
The weights transform as follows. Since
\[
 W(R)L(R)=\rho(1+R^p)^{-\beta-2\alpha}
 =\rho(1-s^p)^{\beta+2\alpha}
\]
and
\[
 \frac{\dd R}{\dd s}=(1-s^p)^{-\alpha-1},
\]
we have, at endpoint roots,
\begin{align*}
 \frac{W(R)L(R)}{|(L+D)'(R)|}
 &=
 \frac{(1-s^p)^{\beta+\alpha-1}}{|h'(s)|},\\
 \frac{W(R)L(R)}{|(L-D)'(R)|}
 &=
 \frac{(1-s^p)^{\beta+\alpha-1}}{|g'(s)|}.
\end{align*}
For the bulk term,
\[
 W(R)D(R)\,\dd R
 =\rho A s(1-s^p)^{\beta+\alpha-1}\,\dd s.
\]
Thus, at every common regular value, the density of the signed
measure in \eqref{eq:app-dist-measure} is $\rho$ times the dual-layer
expression in Theorem~\ref{thm:app-dlt}, with parameter $A=c/\rho$
and threshold $y=x/\rho$. Indeed, $h'(s)>0$, while the sign of $g'(s)$ is the sign of
$s^{p-1}-A$. Thus $h$ is strictly increasing, and $g$ is strictly
monotone on each of at most two intervals. After splitting once more at
the possible zero of $g$, the same is true of $|g|$. On each monotone
branch, the one-dimensional change-of-variables formula shows that the
pushforward of the absolutely continuous endpoint measure is absolutely
continuous. The isolated critical values create no atoms. Consequently,
at almost every $x>0$ the Radon--Nikodym density is given by the coarea
sums above. The dual layer inequality shows that this density, and hence
the signed measure $xF''-F'$, is nonnegative. By the change of variables $q=x^2$, this is
equivalent to $G''\ge0$ distributionally on $(0,\infty)$. Continuity
at $0$ extends convexity to $[0,\infty)$.

Since $Y_1^2\cx Y_2^2$ and $G$ is convex,
\[
 \E G(Y_1^2)\le\E G(Y_2^2),
\]
which is the desired lower-hinge inequality after appending the coordinate.
This proves the theorem for $1/2<\alpha<1$.

It remains to treat $\alpha=1/2$. Choose $\alpha_j\downarrow1/2$ and
set
\[
 \beta_j=\max\{\beta,1+2\alpha_j\}.
\]
Then $\beta_j\to\beta$ and
$T_j\sim\operatorname{Beta}(\alpha_j,\beta_j)$ converges in law to
$T\sim\operatorname{Beta}(1/2,\beta)$. Apply the already proved case
to $(\alpha_j,\beta_j)$ and test with a lower hinge
$q\mapsto(\rho^2-q)_+$. Its value is bounded by $\rho^2$, so the
corresponding expectations converge as $j\to\infty$; the transformed
variables converge in distribution because the beta laws and the
exponents converge. The limiting lower-hinge inequality is the desired
one at $\alpha=1/2$. Equality of the transformed second moments follows
directly from $\E Y_1^2=\E Y_2^2$ and the symmetry of $\eps$. The
lower-hinge criterion completes the proof at the endpoint.
\end{proof}
\section{Proof of the main theorem and strictness}\label{sec:main-proof}

We now combine the two-dimensional comparison with the preservation
theorem. Throughout the first proof below, $1<p<2$ and
$\alpha=1/p$.

\begin{proof}[Proof of Theorem~\ref{thm:convex-order-main} for $1<p<2$]
The case $n=1$ is immediate. By unconditionality of $B_p^n$, signs of
the coefficients can be absorbed into the Rademacher variables in the
sign--Dirichlet representation, so we may work with nonnegative
coefficients. It is enough to prove the claim for one elementary
$T$-transform, since finite compositions of such transforms generate
majorization \cite[Ch.~2]{MarshallOlkinArnold} and convex order is
transitive.

Permute coordinates so the transform acts on coordinates $1,2$, while
the remaining coordinates are common. Proposition~\ref{prop:app-base2}
gives the squared convex order for the two active coordinates alone.
Append the common remaining coordinates one at a time. If the current
dimension is $d\ge2$, Lemma~\ref{lem:app-append-distribution} writes the
append operation as
\[
 Y\longmapsto(1-T)^\alpha Y+cT^\alpha\eps,
 \qquad
 T\sim\operatorname{Beta}(\alpha,1+d\alpha).
\]
Since $1+d\alpha\ge1+2\alpha$,
Theorem~\ref{thm:app-append} preserves squared convex order. After all
common coordinates have been appended, the elementary transform is
proved in the original dimension. Iterating over the transforms proves
\eqref{eq:convex-order-main}.
\end{proof}

\begin{proof}[Proof of Theorem~\ref{thm:convex-order-main} for $p=1$]
Let $p_j\downarrow1$, and let $X_{p_j}$ be uniform on $B_{p_j}^n$.
The bodies $B_{p_j}^n$ converge to $B_1^n$ in measure and their volumes
converge, hence the uniform measures converge weakly to the uniform
measure on $B_1^n$. All projections involved are supported in a common
compact interval. Therefore, for every continuous convex function
$\Phi$ on that interval,
\[
 \E\Phi(\langle b,X_{p_j}\rangle^2)
 \le
 \E\Phi(\langle a,X_{p_j}\rangle^2)
\]
passes to the limit. This proves \eqref{eq:convex-order-main} for
$p=1$.
\end{proof}

\begin{proof}[Proof of Theorem~\ref{thm:convex-order-main} for $p=2$]
Majorization of $\sqvec a$ over $\sqvec b$ implies
$\|a\|_2=\|b\|_2$. The uniform measure on $B_2^n$ is invariant under
the orthogonal group, so $\langle a,X\rangle$ and
$\langle b,X\rangle$ have the same law whenever the coefficient
vectors have the same Euclidean norm. Their squares are therefore equal
in distribution, which proves \eqref{eq:convex-order-main} as equality.
\end{proof}

The following calculation is used to obtain strictness at the endpoint
$p=1$.

\begin{lemma}[Fourth moments for the cross-polytope]\label{lem:p1-fourth-moments}
Let $X$ be uniform on $B_1^n$ and let $i\neq j$. Then
\begin{align}\label{eq:p1-m4-m22}
 \E X_i^4
 &=\frac{24}{(n+1)(n+2)(n+3)(n+4)},\notag\\
 \E X_i^2X_j^2
 &=\frac{4}{(n+1)(n+2)(n+3)(n+4)}.
\end{align}
Consequently, for every $\theta\in\R^n$,
\begin{equation}\label{eq:p1-fourth-moment-projection}
 \E\langle\theta,X\rangle^4
 =
 3\,\E X_1^2X_2^2\,\|\theta\|_2^4
 +\bigl(\E X_1^4-3\,\E X_1^2X_2^2\bigr)
 \sum_{i=1}^n\theta_i^4,
\end{equation}
and the coefficient $\E X_1^4-3\,\E X_1^2X_2^2$ is strictly positive.
\end{lemma}

\begin{proof}
In the positive orthant of $B_1^n$, set $T_i=x_i$ and let $T_0=1-\sum_{i=1}^n T_i$. The Jacobian is one, and the simplex coordinates are uniformly distributed. Restoring the independent Rademacher signs gives
\[
 X_i=\eps_i T_i,
 \qquad
 (T_0,T_1,\ldots,T_n)\sim \operatorname{Dirichlet}(1,1,\ldots,1),
\]
with the signs independent of $T$. In particular, $X_i^2=T_i^2$ and $X_i^2X_j^2=T_i^2T_j^2$.
For a Dirichlet$(a_0,\ldots,a_n)$ vector, the mixed moments satisfy
\[
 \E\prod_{k=0}^n T_k^{m_k}
 =
 \frac{\Gamma(\sum_{\ell=0}^n a_\ell)}{\Gamma(\sum_{\ell=0}^n a_\ell+\sum_{\ell=0}^n m_\ell)}
 \prod_{k=0}^n\frac{\Gamma(a_k+m_k)}{\Gamma(a_k)}.
\]
Here all $a_k=1$ and $\sum_{\ell=0}^n a_\ell=n+1$. Taking first $m_i=4$ and all other $m_k=0$, and then taking $m_i=m_j=2$, gives the two identities in \eqref{eq:p1-m4-m22}.

Finally, expanding $\langle\theta,X\rangle^4$ and using sign symmetry to eliminate odd mixed moments yields
\[
 \E\langle\theta,X\rangle^4
 =\sum_{i=1}^n\theta_i^4 \E X_i^4
 +6\sum_{1\le i<j\le n}\theta_i^2\theta_j^2\,\E X_i^2X_j^2,
\]
which simplifies to \eqref{eq:p1-fourth-moment-projection} because
\[
 2\sum_{1\le i<j\le n}\theta_i^2\theta_j^2
 =\|\theta\|_2^4-\sum_{i=1}^n\theta_i^4.
\]
The coefficient equals $12/((n+1)(n+2)(n+3)(n+4))>0$.
\end{proof}

\begin{proposition}[Strictness for strictly convex tests]\label{prop:strict-convex-tests}
Let $1\le p<2$, let $X$ be uniform on $B_p^n$, and suppose
$\sqvec a\succ\sqvec b$, with $\sqvec a$ not a permutation of
$\sqvec b$. If $\Phi$ is strictly convex on an interval containing the
supports of the two squared projections, then
\[
 \E\Phi(\langle b,X\rangle^2)
 <
 \E\Phi(\langle a,X\rangle^2).
\]
\end{proposition}

\begin{proof}
Theorem~\ref{thm:convex-order-main} gives the non-strict inequality.
By Strassen's martingale characterization of convex order, equality for
one strictly convex test function would force the two random variables
to have the same law; see \cite{Strassen} and, for a modern treatment,
\cite[Ch.~3.A]{ShakedShanthikumar}. It therefore remains to show that
the squared-projection laws are different.

Assume first that $1<p<2$, and let $q=p/(p-1)>2$. The support of
$\langle a,X\rangle$ is
\[
 [-\|a\|_q,\|a\|_q].
\]
Since $\sqvec a\succ\sqvec b$ and the vectors are not permutations,
strict Karamata inequality for $x\mapsto x^{q/2}$ gives
\[
 \sum_{i=1}^n |a_i|^q
 >
 \sum_{i=1}^n |b_i|^q.
\]
Thus the two projection laws have different supports.

For $p=1$, Lemma~\ref{lem:p1-fourth-moments} shows that, among
equal-norm coefficient vectors,
$\E\langle a,X\rangle^4$ is an affine function of
$\sum_i a_i^4$ with strictly positive slope. Strict Karamata inequality
for $x\mapsto x^2$ gives
$\sum_i a_i^4>\sum_i b_i^4$, so the fourth moments, and hence the laws,
are different.
\end{proof}

\begin{remark}[Degeneracy at the Euclidean endpoint]\label{rem:p-endpoints}
At $p=2$, every comparison in Theorem~\ref{thm:convex-order-main} is
an equality. Consequently, Proposition~\ref{prop:strict-convex-tests}
and all later strict directional statements are specific to $p<2$.
For $p>2$, the coordinate versus two-coordinate example in the
introduction shows that the same convex-order direction is false.
\end{remark}

\section{Extensions and integral consequences}\label{sec:extensions}

\subsection{Scale mixtures and radially decreasing measures}

\begin{proposition}[Scale-mixture stability]\label{prop:scale-mixture}
Let $n\ge1$ and $1\le p\le2$, let $X$ be uniform on $B_p^n$, and let
$R\ge0$ be
independent of $X$ with $\E R^2<\infty$. If $\sqvec a\succ\sqvec b$,
then
\[
 \langle b,RX\rangle^2\cx\langle a,RX\rangle^2.
\]
\end{proposition}

\begin{proof}
Let $\Phi$ be convex and assume the relevant expectations are finite.
For every fixed $r\ge0$, the function $u\mapsto\Phi(r^2u)$ is convex.
Theorem~\ref{thm:convex-order-main} therefore gives
\[
 \E_X\Phi\!\left(r^2\langle b,X\rangle^2\right)
 \le
 \E_X\Phi\!\left(r^2\langle a,X\rangle^2\right).
\]
Integrating this inequality with respect to the law of $R$ proves the
claim. Equality of the means follows from
$\|a\|_2=\|b\|_2$ and $\E R^2<\infty$.
\end{proof}

\begin{corollary}[Radially decreasing densities]\label{cor:radial-density}
Let $\rho:[0,\infty)\to[0,\infty)$ be nonincreasing and
right-continuous, and assume that
\[
 n|B_p^n|\int_0^\infty r^{n-1}\rho(r)\,\dd r=1
\]
and that the probability measure
\[
 \mu(\dd x)=\rho(\|x\|_p)\,\dd x
\]
has finite second moment. If $1\le p\le2$ and $\sqvec a\succ\sqvec b$,
then, for $Y\sim\mu$,
\[
 \langle b,Y\rangle^2\cx\langle a,Y\rangle^2.
\]
\end{corollary}

\begin{proof}
Let $\nu$ be the positive Stieltjes measure determined by
\[
 \nu((s,t])=\rho(s)-\rho(t),
 \qquad 0\le s<t<\infty,
\]
and define
\[
 \pi(\dd r)=|B_p^n|\,r^n\,\nu(\dd r).
\]
Integration by parts gives
\[
 \pi((0,\infty))
 =
 n|B_p^n|\int_0^\infty r^{n-1}\rho(r)\,\dd r
 =1,
\]
where the boundary terms vanish by monotonicity and integrability.
At every continuity point $s$ of $\rho$ (and therefore for almost
every $s\ge0$),
\[
 \rho(s)
 =
 \int_{(s,\infty)}
 \frac{1}{|B_p^n|r^n}\,\pi(\dd r).
\]
Consequently,
\[
 \mu
 =
 \int_0^\infty \operatorname{Unif}(rB_p^n)\,\pi(\dd r).
\]
Equivalently, $Y$ has the same law as $RX$, where $X$ is uniform on
$B_p^n$ and $R$ has law $\pi$, independently of $X$. Moreover,
\[
 \E\|Y\|_2^2=\E R^2\,\E\|X\|_2^2.
\]
The finite second-moment assumption therefore gives $\E R^2<\infty$,
and Proposition~\ref{prop:scale-mixture} applies.
\end{proof}

\begin{remark}
The corollary includes, for example, densities proportional to
$e^{-V(\|x\|_p)}$ whenever $V$ is nondecreasing and the resulting law
has finite second moment. It also includes compactly supported
nonincreasing radial profiles on $B_p^n$.
\end{remark}

\subsection{Moments, hinge functions, and Laplace transforms}

\begin{corollary}[Integral comparisons]\label{cor:moments-laplace}
Let $1\le p\le2$, let $X$ be uniform on $B_p^n$, and suppose
$\sqvec a\succ\sqvec b$. Then
\begin{align}
 \E|\langle b,X\rangle|^r
 &\le
 \E|\langle a,X\rangle|^r,
 && r\in[2,\infty),\label{eq:moment-convex}\\
 \E|\langle b,X\rangle|^r
 &\ge
 \E|\langle a,X\rangle|^r,
 && 0<r<2,\label{eq:moment-concave}\\
 \E e^{-\lambda\langle b,X\rangle^2}
 &\le
 \E e^{-\lambda\langle a,X\rangle^2},
 && \lambda>0.\label{eq:laplace-order}
\end{align}
If $a\ne0$ (and hence $b\ne0$), then also
\begin{align}
 \E|\langle b,X\rangle|^r
 &\le
 \E|\langle a,X\rangle|^r,
 && -1<r<0,\label{eq:negative-moment}\\
 \E\log|\langle b,X\rangle|
 &\ge
 \E\log|\langle a,X\rangle|.\label{eq:log-moment}
\end{align}
In addition, for every $t\ge0$,
\[
 \E\bigl(\langle b,X\rangle^2-t\bigr)_+
 \le
 \E\bigl(\langle a,X\rangle^2-t\bigr)_+.
\]
If $1\le p<2$ and $\sqvec a$ is not a permutation of
$\sqvec b$, then \eqref{eq:laplace-order} is strict for every
$\lambda>0$. If $p=2$, all the displayed comparisons are equalities.
\end{corollary}

\begin{proof}
Apply Theorem~\ref{thm:convex-order-main} to the functions
$u\mapsto u^{r/2}$. This function is convex for
$r\in(-\infty,0)\cup[2,\infty)$ and concave for $0<r<2$. For
$-1<r<0$, use the convex approximants
$u\mapsto(u+\varepsilon)^{r/2}$ and let $\varepsilon\downarrow0$.
When $a\ne0$, the two projections are nondegenerate. Their laws have
log-concave densities on the interiors of their supports, hence densities
that are locally bounded near the origin; therefore the negative moments
are finite for $r>-1$. The logarithmic comparison
follows from concavity of $\log(u+\varepsilon)$, followed by truncation
and the local integrability of the logarithm against those marginal
densities. The hinge inequality is the standard characterization of
convex order, and \eqref{eq:laplace-order} follows because
$u\mapsto e^{-\lambda u}$ is convex. For $1\le p<2$, strictness follows from
Proposition~\ref{prop:strict-convex-tests}; for $p=2$, equality in law
was proved in Theorem~\ref{thm:convex-order-main}.
\end{proof}

\section{Gaussian convolution and central sections}\label{sec:applications}

Let $\nu_{p,n}$ denote normalized Lebesgue measure on $B_p^n$, and let
$\gamma_t$ be the centered Gaussian measure with covariance $tI_n$.
For $t\ge0$, put
\[
 \mu_{p,n,t}=\nu_{p,n}*\gamma_t.
\]
If $Y\sim\mu_{p,n,t}$ and $h_{p,n,t,\theta}$ denotes the density of
$\langle Y,\theta\rangle$ for $\theta\in\sphere$, define
\[
 A_{p,n,t}(\theta)=h_{p,n,t,\theta}(0),
 \qquad
 \widetilde A_{p,n,t}(\theta)
 =
 \sqrt{\Var\langle Y,\theta\rangle}\,
 h_{p,n,t,\theta}(0).
\]
The standardization is independent of the direction on the Euclidean
sphere.

The volume of $B_p^n$ is
\begin{equation}\label{eq:volume-bp}
 |B_p^n|=\frac{(2\Gamma(1+1/p))^n}{\Gamma(1+n/p)}.
\end{equation}

\begin{lemma}[Coordinate moments]\label{lem:moments}
Let $p>0$, $n\ge1$, $X\sim\nu_{p,n}$, and let $r>-1$. Then
\begin{equation}\label{eq:coordinate-moment}
 \E |X_1|^r
 =\frac{\Gamma(1+n/p)\Gamma((r+1)/p)}
 {\Gamma(1/p)\Gamma(1+(n+r)/p)}.
\end{equation}
In particular,
\begin{equation}\label{eq:variance-vpn}
 v_{p,n}:=\E X_1^2
 =\frac{\Gamma(3/p)\Gamma(1+n/p)}
 {\Gamma(1/p)\Gamma(1+(n+2)/p)}.
\end{equation}
Therefore
\[
 \Var \langle X,\theta\rangle=v_{p,n}\lVert\theta\rVert_2^2.
\]
\end{lemma}

\begin{proof}
The density of $X_1$ is proportional to $(1-|u|^p)_+^{(n-1)/p}$. Normalizing by \eqref{eq:volume-bp} gives
\[
 f_{p,n}(u)=\frac{|B_p^{n-1}|}{|B_p^n|}(1-|u|^p)_+^{(n-1)/p}.
\]
The substitution $z=u^p$ gives \eqref{eq:coordinate-moment}. The covariance is diagonal by sign symmetry and has equal diagonal entries by permutation symmetry.
\end{proof}

For $t>0$ define
\[
 M_{p,n,t}(\theta)=\E\exp\left(-\frac{\langle X,\theta\rangle^2}{2t}\right),
 \qquad X\sim\nu_{p,n}.
\]
Equivalently, for $\lambda>0$,
\[
 L_{p,n,\lambda}(\theta)=\E e^{-\lambda\langle X,\theta\rangle^2},
\]
so that $M_{p,n,t}(\theta)=L_{p,n,1/(2t)}(\theta)$.

\begin{proposition}[Smoothing identity]\label{prop:smoothing}
Let $1\le p<\infty$, $n\ge1$, $t>0$, and $\theta\in\sphere$. Then
\begin{equation}\label{eq:smoothing-identity}
 \widetilde A_{p,n,t}(\theta)
 =\frac1{\sqrt{2\pi}}\sqrt{1+\frac{v_{p,n}}{t}}\,
 \E\exp\left(-\frac{\langle X,\theta\rangle^2}{2t}\right),
 \qquad X\sim\nu_{p,n}.
\end{equation}
Consequently, for fixed $p,n,t$, optimizing $\widetilde A_{p,n,t}$ over directions is equivalent to optimizing $M_{p,n,t}$.
\end{proposition}

\begin{proof}
Let $G$ be a standard Gaussian vector independent of $X\sim\nu_{p,n}$. Then the projection of $X+\sqrt tG$ in direction $\theta$ is
\[
 \langle X,\theta\rangle+\sqrt t Z,
\]
where $Z\sim N(0,1)$ is independent. Therefore its density at zero is
\[
 h_{p,n,t,\theta}(0)
 =\frac1{\sqrt{2\pi t}}
 \E\exp\left(-\frac{\langle X,\theta\rangle^2}{2t}\right).
\]
Multiplying by the standard deviation $\sqrt{v_{p,n}+t}$ gives \eqref{eq:smoothing-identity}.
\end{proof}

\begin{corollary}[Positive-time Schur-convexity]\label{cor:positive-time-schur}
Let $1\le p\le2$ and $t>0$. The map
\[
 \theta\longmapsto\widetilde A_{p,n,t}(\theta),
 \qquad \theta\in\sphere,
\]
is Schur-convex as a function of $\sqvec\theta$. More precisely, if
$\sqvec\theta\succ\sqvec\eta$, then
\[
 \widetilde A_{p,n,t}(\theta)
 \ge
 \widetilde A_{p,n,t}(\eta).
\]
For $1\le p<2$, the inequality is strict unless
$\sqvec\theta$ is a permutation of $\sqvec\eta$. For $p=2$, the
profile is constant in $\theta$, and every such comparison is an
equality. The same assertions hold for the unstandardized profile
$A_{p,n,t}$.
\end{corollary}

\begin{proof}
Proposition~\ref{prop:smoothing} expresses the profile, up to a positive
factor independent of the direction, as
\[
 \E\exp\left(-\frac{\langle X,\theta\rangle^2}{2t}\right).
\]
Apply \eqref{eq:laplace-order} with $\lambda=1/(2t)$. The strictness
statement for $p<2$ and the equality statement for $p=2$ follow from
Corollary~\ref{cor:moments-laplace}.
\end{proof}

At $t=0$, the density at the origin is the normalized central-section
volume:
\begin{equation}\label{eq:section-density}
 A_{p,n,0}(\theta)
 =
 \frac{\vol_{n-1}(B_p^n\cap\theta^\perp)}{|B_p^n|}.
\end{equation}
The non-strict majorization inequality at $t=0$ follows directly from
Theorem~\ref{thm:convex-order-main}. Indeed, for each direction the
one-dimensional marginal density is continuous at the origin, and the
Gaussian kernels form an approximate identity. Hence
\[
 A_{p,n,t}(\theta)\longrightarrow A_{p,n,0}(\theta)
 \qquad\text{as }t\downarrow0.
\]
Taking the limit in Corollary~\ref{cor:positive-time-schur} proves
Schur-convexity of the central-section function. Thus the classical
non-strict section comparison is a singular consequence of the
convex-order theorem.

For $1\le p<2$, the exact equality cases at $t=0$ are supplied by
the classical section theorem of Meyer--Pajor and Koldobsky, in its
Schur-convex and stability forms due to Eskenazis and
Chasapis--Nayar--Tkocz
\cite{MeyerPajor,Koldobsky1998,KoldobskyBook,Eskenazis,ChasapisNayarTkocz};
see also the survey \cite{NayarTkoczSurvey}. At $p=2$, rotational
invariance makes the central-section function constant, so every
direction is an equality case.

\begin{corollary}[Coordinate and diagonal extremizers]\label{cor:all-time-extremizers}
Let $1\le p\le2$, $n\ge2$, and $t\ge0$. Put
\[
 \diag=\frac1{\sqrt n}(1,\ldots,1).
\]
Then, for every $\theta\in\sphere$,
\[
 \widetilde A_{p,n,t}(\diag)
 \le
 \widetilde A_{p,n,t}(\theta)
 \le
 \widetilde A_{p,n,t}(e_1).
\]
The same inequalities hold for $A_{p,n,t}$. If $1\le p<2$, equality
occurs only on the corresponding signed-permutation orbits, both for
$t>0$ and for $t=0$. If $p=2$, both inequalities are equalities for
every $\theta\in\sphere$ and every $t\ge0$.
\end{corollary}

\begin{proof}
The vector $(1,0,\ldots,0)$ majorizes every point of the simplex, and
every point of the simplex majorizes $(1/n,\ldots,1/n)$. Apply Corollary~\ref{cor:positive-time-schur} when $t>0$ and the
preceding $t=0$ discussion when $t=0$. These results also give the
stated equality cases in the regimes $p<2$ and $p=2$. Since
$\Var\langle Y,\theta\rangle=v_{p,n}+t$ for every
$\theta\in\sphere$, standardization multiplies all directions by the
same positive factor.
\end{proof}

\section{Further remarks and open problems}\label{sec:discussion}

Theorem~\ref{thm:convex-order-main} identifies a distributional
monotonicity principle for projections of $B_p^n$ in the range
$1\le p\le2$. It is nontrivial for $p<2$ and reduces to equality at
the Euclidean endpoint. The section and heat-flow inequalities are applications of
this principle rather than separate inputs: all positive smoothing
times follow from the Laplace-transform test, while the non-strict
central-section inequality is recovered by an approximate-identity
limit. Corollary~\ref{cor:radial-density} shows that the mechanism also
applies to a broad class of $\ell_p$-radial measures.

Several questions remain. First, it would be useful to characterize
the radial profiles, beyond common scale mixtures, for which the same
convex order is valid. Second, a quantitative version should measure
the convex-order deficit in terms of a distance between
$\sqvec a$ and $\sqvec b$ under suitable uniform convexity assumptions
on the test function. Finally, the support obstruction for $p>2$
shows that the same convex order cannot describe the phase-transition
regime. A natural problem is to identify a weaker or
parameter-dependent stochastic comparison that captures the
directional changes occurring there.

\backmatter

\end{document}